\documentclass{article}
\usepackage[utf8]{inputenc}
\usepackage[english]{babel}
\usepackage{csquotes}
\usepackage[utf8]{inputenc}
\usepackage{amsmath}
\usepackage{graphicx}
\usepackage{amssymb}
\usepackage{eucal}
\usepackage{amsthm}
\usepackage{tikz-cd}
\usepackage{mathrsfs}
\usepackage[colorinlistoftodos]{todonotes}
\usepackage{enumitem}
\usepackage{yfonts}
\usepackage{mathtools}
\usepackage{biblatex}

\newtheorem{thm}{Theorem}[section]
\newtheorem{lem}[thm]{Lemma}

\newtheorem{defn}[thm]{Definition}
\newtheorem{eg}[thm]{Example}

\newtheorem{conj}[thm]{Conjecture}
\newtheorem{cor}[thm]{Corollary}

\newtheorem{rmk}[thm]{Remark}
\newtheorem{prop}[thm]{Proposition}

\newcommand{\C}{\mathbb{C}}
\newcommand{\Z}{\mathbb{Z}}

\newcommand{\K}{\mathbb{K}}

\newcommand{\miu}[2]{\mu_{#1|1|#2}}

\newcommand{\mcA}{\mathcal{A}}
\newcommand{\mcB}{\mathcal{B}}

\newcommand{\mcM}{\mathcal{M}}
\newcommand{\mcF}{\mathcal{F}}
\newcommand{\mcY}{\mathcal{Y}}

\newcommand{\bimod}[2]{\mcY^l_{#1}\otimes_{\mathcal{\K}}\mcY^r_{#2}}
\newcommand{\lmod}[1]{\mcY^l_{#1}}
\newcommand{\rmod}[1]{\mcY^r_{#1}}

\newcommand{\gr}[1]{\mcG_{#1}}

\newcommand{\ta}[1]{\tau_{#1}}
\newcommand{\ti}[1]{\tau^{-1}_{#1}}

\newcommand{\bihom}[2]{hom_{\mcF-\mcF}(#1,#2)}

\newcommand{\biHom}[2]{Hom_{\mcF-\mcF}(#1,#2)}

\newcommand{\mcX}{\mathcal{X}}
\newcommand{\mcZ}{\mathcal{Z}}
\newcommand{\mcE}{\mathcal{E}}
\newcommand{\mcN}{\mathcal{N}}
\newcommand{\mcG}{\mathcal{G}}
\newcommand{\mcL}{\mathcal{L}}
\newcommand{\mcT}{\mathcal{T}}

\newcommand{\mcnN}[1]{\mathcal{N}^{#1}}
\newcommand{\bfA}{\mathbf{A}}
\newcommand{\bfB}{\mathbf{B}}

\newcommand{\nbf}[1]{\prescript{}{#1}{\mathbf{n}}^{r|1|s}}
\newcommand{\bfnn}[1]{\mathbf{n}_{#1}^{r|1|s}}
\newcommand{\bfnnn}[2]{ \prescript{}{#1}{\mathbf{n}}_{#2}^{r|1|s}}

\newcommand{\bfnnnz}[2]{ \prescript{}{#1}{\mathbf{n}}_{#2}^{0|1|0}}

\newcommand{\summand}{i_1\cdots i_k}
\newcommand{\lsubsummand}{i_{l'}\cdots i_k}
\newcommand{\rsubsummand}{i_1\cdots i_{k-l}}
\newcommand{\msubsummand}{i_1\cdots i_{l-1}i_{l'+1}\cdots i_k}

\newcommand{\secondsum}{1\leq i_1\leq\cdots i_k\leq n}

\newcommand{\lthirdsum}{0< l'< k}
\newcommand{\rthirdsum}{0< l< k}
\newcommand{\mthirdsum}{1<l<l'< k}
\newcommand{\elements}{c_1,\cdots,c_{m+1}}

\newcommand{\summandnew}{i_1\cdots i_m}
\newcommand{\subsummandnew}{j_1\cdots j_{m'}}

\def\bimodc{\mathrm{\!-\! mod\!-\!}}

\begin{document}
\title{A long exact sequence on the composition of Dehn twists}
\author{Shuo Zhang}
\maketitle

\begin{abstract}
We prove an exact sequence relating the Lagrangian Floer homology of a collection of Lagrangian spheres $\{L_i\}$ and the fixed-point
Floer homology of iterated Dehn twists along them, making progress toward a conjecture of Seidel. 

\end{abstract}

\section{Introduction}
It is a classical result that for an isolated hypersurface singularity $$f:\C^{n+1}\rightarrow \C$$ the nearby smooth fibers are homotopy equivalent to a bouquet of spheres $S^n$ called \textit{vanishing cycles}. Moreover, the monodromy around the singularity is a composition of Dehn twists along these vanishing cycles \cite{Arnold2012}. In fact, these structures are compatible with the Weinstein structure on the total space $\C^{n+1}$ (with some minor modifications). For example, the vanishing cycles are exact Lagrangian submanifolds, and the Dehn twists are exact symplectomorphisms. It is therefore natural to study their symplectic invariants, such as Lagrangian Floer homology, Fukaya categories, and fixed-point Floer homology.

 Given a finite collection of (monotone or exact) Lagrangian spheres $\{L_i\}$ in a (monotone or exact) symplectic manifold $(M,\omega)$, the (compact) Fukaya category $\mcB$ generated by $\{L_i\}$ is defined to have $\{L_i\}$ as objects and $cF(L_i,L_j)$ as morphisms. In the case of Lefschetz fibrations, the order of the vanishing cycles is also important, which motivates Seidel \cite{Sei00} to define the directed Fukaya category $\mcA$:
$$\operatorname{hom}_{\mathcal{A}}\left(L_j, L_k\right)= \begin{cases}CF(L_j, L_k) & j<k, \\ \mathbb{K} e_j & j=k, \\ 0 & j>k\end{cases}$$

In this paper, we work in a $\mathbb{Z}$-graded setting. Let $(M,\omega)$ be a symplectic manifold of dimension $2d$ with $2c_1(M,\omega)=0$ and $\{L_i\}$ be graded Lagrangian spheres. Seidel conjectured that the fixed-point Floer homology of the composed Dehn twists $\tau:=\ta{1}\circ\cdots\circ\ta{n}$ can be computed from the topology of $M$ and all of $\operatorname{hom}_{\mathcal{A}}\left(L_j, L_k\right)$. To be more precise:
\begin{conj}[{\autocite[Conjecture 3]{Sei06}}]\label{conjecture}
There is a long exact sequence:
\begin{align}\label{Seidel exact sequence}
\cdots \rightarrow H^*(D) \rightarrow HF^*(id) \rightarrow HF^*(\tau ) \rightarrow \ldots
\end{align}
Where $D$ is the following cochain complex:
\begin{align}\label{seidel seq}
    \bigoplus_{1\leq k\leq n}\bigoplus_{1\leq i_1< \cdots< i_k\leq n}CF(L_{i_k},L_{i_1})\otimes CF(L_{i_{k-1}}, L_{i_k})\cdots\otimes CF(L_{i_1}, L_{i_2})[d+n]
\end{align}
with a differential analogous to the Hochschild differential.
\end{conj}
\begin{rmk}
The differential has a more natural description, again due to Seidel.  Consider $\mcA$ and $\mcB$ as an $A_\infty$ algebra over the semi-simple ring $R=\K e_1\oplus\cdots\oplus \K e_n$ and $\mcA_+$ to be the kernel of the augmentation map $\mcA\rightarrow R$. Then (\ref{seidel seq}) is the chain complex 
\begin{align}
    (\mcB[d]\otimes_RT(\mcA_+[1])^{diag}
\end{align}
which is exactly $CC_*(\mcA,\mcB[d])$. See \autocite[Section 2.9]{Ganatra}
\end{rmk}

Let $\mcF:=Fuk(M)$ be the compact Fukaya category of $M$. In this paper, we construct an exact triangle in the category of $\mcF$-bimodules that we expect to give the exact sequence (\ref{Seidel exact sequence}) when we apply the bimodule $Hom$ functor.
We fixed the notations.  We use $\tau_i$ to denote $\tau_{L_i}$. The left and right Yoneda modules are denoted as $\mcY_X^l(Y):=\operatorname{hom}(X, Y)$  and $\mcY_X^r(Y):=\operatorname{hom}(Y, X)$ respectively. $\mcF_\Delta$ and $\mathfrak{G}_\tau$ are the diagonal and graph bimodule of $\tau$ respectively. All chain complexes are over $\Z/2$, though we expect everything to work over $\Z$.

The following is the main theorem. 
\begin{thm}[Main theorem]\label{mainmain}
Let $(M,\omega)$ be a symplectic manifold with $2c_1(M)=0$ with a collection of Lagrangian spheres $\{L_i\}$ such that their Floer cohomologies are well defined. There is an exact triangle of $Fuk(M)$-bimodules
$$\mcE_n\xrightarrow{\widetilde{ev}} \mcF_{\Delta} \rightarrow \mathfrak{G}_\tau$$
where  $$\mcE_n(A,B)=\bigoplus_{1\leq k\leq n}\bigoplus_{1\leq i_1< \cdots< i_k\leq n}CF(A,L_{i_1})\otimes\cdots\otimes CF(L_{i_k},B)[k-1]$$ is a bimodule version of the bar complex $T\mcA$. Its structure maps consists of all possible contractions by the $A_\infty$ operations $\mu_i$ and $\widetilde{ev}$ is the map 
\begin{align}
    \mu_{k+1}: CF(A,L_{i_1})\otimes\cdots\otimes CF(L_{i_k},B)[k-1] \rightarrow (A,B)
\end{align}
on each direct summand.
\end{thm}

\begin{eg}
    When $n=1$, the exact triangle reduces to the exact triangle $$\bimod{L_1}{L_1}\xrightarrow{ev} \mcF_\Delta\rightarrow \mathfrak{G}_{\tau_1}$$
    which is the bimodule version of the exact triangle in \cite{Mak-Wu}
\end{eg}

\begin{rmk}
    We point out that Sikimeti Ma'u and Tim Perutz announced a proof of Conjecture \ref{conjecture} as early as 2011 using a direct approach via quilted Floer theory.  Although our approach also relies on M'au-Wehrheim-Woodward's Lagrangian correspondence theory, the main tool we use is the Lagrangian cobordism due to Biran and Cornea \cite{BC1}\cite{BC2}, as well as techniques from \cite{Mak-Wu}. It is also worth pointing out that the Biran and Cornea generalized their work to Lagrangian cobordism in Lefschetz fibrations \cite{BC3}.
\end{rmk}

To relate this exact triangle in the category of bimodules over $\mcF$, we consider the category $\mcF^*(M\times M^-)$ which is the pre-triangulated split closure of the compact Fukaya category of $M\times M^-$ together with the diagonal $\Delta$.  we use the subcategory $\mcF^2$ of the category $\mcF^*(M\times M^-)$ which is generated by products of compact Lagrangians.  
\begin{rmk}
    When $M$ is exact and therefore noncompact, the category $\mcF^*(M\times M^-)$ needs some clarification since the diagonal is not a compact Lagrangian. The morphism between a compact Lagrangian $A$ and a non-compact Lagrangian $B$ can be defined without difficulties. The morphism between two noncompact Lagrangians $(e.g. A=B=\Delta)$ is defined by perturbing $B$ using a small \textit{positive} Hamiltonian flow that is equal to the Reeb flow near the boundary.
\end{rmk}
It is an algebraic fact that the functor $\mathbf{M}:\mcF^2\rightarrow {\mcF}\bimodc{\mcF}$ that maps $A\times B$ to $\bimod{A}{B}$ is full \autocite[Proposition 1.2]{Ganatra}. When $M$ is nondegenerate, $\Delta$ (and hence the graph of any symplectomorphism) is also contained in $\mcF^2$, and $\mathbf{M}$ is full on these objects too. This implies:
\begin{align}
    \biHom{\mcF_\Delta}{\mcF_\Delta}&\cong HF(\Delta,\Delta,+) \cong HF(id,+)\\
    \biHom{\mcF_\Delta}{\mcF_\phi}&\cong HF(\Phi,+)\\
    \bihom{N\times N'}{\Delta}&\cong HF(N,N')
\end{align} 
\begin{rmk}
    We explain what the $\pm$ means. When $M$ is exact, which means that it is not compact, one has to perturb the symplectomorphism $\phi$ near the boundary to define $HF(\phi,\pm)$. This is achieved by picking a Hamiltonian $H$ such that $H_{|\partial M}\equiv 0$ and the Hamiltonian flow $\psi^H_t$ equal to the Reeb flow. Then the perturbation $\phi\circ\psi^H_t$ defines $HF(\phi,+)$ when $t>0$ and $HF(\phi,-)$ when $t<0$. When $M$ is non-exact and closed, then we can ignore the "$+$" and "$-$" as the Floer homologies do not depend on the perturbation.
\end{rmk}

Now taking $\biHom{\mcF_\Delta}{-}$ and the isomorphisms we have:
\begin{cor}\label{main conjecture}
If $M$ and $\{L_i\}$ are exact or monotone, and the compact Fukaya category is nondegenerate, then there's a long exact sequence

        \begin{align}
        \cdots \rightarrow  H^*(D) \rightarrow H F^*(id,+) \rightarrow  H F^*(\tau, +)\rightarrow \ldots
        \end{align}
    where $D$ is the chain complex $\biHom{\mcE_n}{\mcF_\Delta}$. See \autocite[Definition 2.18]{Ganatra}.
\end{cor}
\begin{rmk}
    If in our definition of $\mcF^*(M\times M^-)$ we use the negative Hamiltonian to perturb the Lagrangians near infinity, we would get similar results with "$+$" replaced by "$-$"
\end{rmk}

Given two Lagrangians $N$ and $N'$, taking $\biHom{N\times N'}{-}$, we have the open version of the exact sequence. 
\begin{cor}\label{open version}
    Under the same hypothesis as Theorem \ref{mainmain}, for every pair of Lagrangians $N$ and $N'$ there is a long exact sequence
    \begin{align}
         \cdots \rightarrow H^*(\mcE_n(N,N')) \rightarrow H F^*(N,N') \rightarrow   HF^*(N,\tau(N') ) \rightarrow \ldots
    \end{align}
\end{cor}
\begin{rmk}\label{bigcomment}
Theorem \ref{mainmain} works for any symplectic manifolds $(X,\omega)$ and graded Lagrangian spheres $\{L_i\}$ that satisfy $2c_1(X,\omega)=0$ since our argument is purely algebraic and uses the $\Z$-grading (though when $d$ is odd a $\Z/2$ grading is enough).
It is possible to use our argument on the category of product Lagrangians directly (without going to the category of bimodules) to produce an object $E_n$ in $\mcF^2$ corresponding to the bimodule $\mcE_n$. However, the chain complex $CF(E_n,\Delta)$ is harder to work with than $\biHom{\mcE_n}{\mcF_\Delta}$. Also one needs to prove a Lagrangian surgery version of Lemma \ref{changecone0}.
\end{rmk}

\textbf{Acknowledgement} I would like to express my gratitude to my co-advisor, Weiwei Wu, for introducing this problem, and for innumerable helpful discussions and suggestions on both mathematics and writing. I also thank Erkao Bao, Octav Cornea, Mark Mclean, Cheuk Yu Mak, and Micheal Usher for their helpful comments on this draft; and my advisor, Tian-Jun Li, for his continuous support and encouragement.

\section{Proof}
First, we recall the definitions of diagonal bimodules, Yoneda modules, and algebraic twists in an $A_\infty$ category $\mcF$, as defined in \cite{Sei08}. For simplicity, we use $(A,B)$ to denote $\hom_\mcF(A,B)$ and $\mu_i$ to denote $\mu_\mcF^i$.
\begin{defn}
    The \textbf{diagonal bimodule} is defined as $\mcF_\Delta(A,B)=(A,B)$ with structure maps:

    \begin{align}
        \mu_\Delta^{k|1|s}:(A_r,A_{r-1})\otimes\cdots\otimes (A_0,B_0)\otimes\cdots\otimes(B_{s-1},B_s)\rightarrow (A_k,B_s)
    \end{align}
    given by $\mu_\Delta^{r|1|s}=\mu_{r+s+1}$
\end{defn}

\begin{defn}
    Given an object X, the \textbf{left Yoneda module} $\lmod{X}$ is defined as $\lmod{X}:=(-,X)$ with structure maps:
    \begin{align}
        \mu_X^{r|1}: (A_r,A_{r-1})\otimes\cdots\otimes(A_0,X) \rightarrow (A_r,X) 
    \end{align}
    given by $ \mu_X^{r|1}=\mu_{r+1}$
\end{defn}

\begin{defn}
    Given an object X, the \textbf{right Yoneda module} $\rmod{X}$ is defined as $\rmod{X}:=(X,-)$ with structure maps:
    \begin{align}
        \mu_X^{1|s}: (X,B_0)\otimes\cdots\otimes(B_{s-1},B_s) \rightarrow (X,B_s) 
    \end{align}
    given by $ \mu_X^{1|s}=\mu_{s+1}$
\end{defn}

\begin{defn}\label{product bimodule}
    Given a left module $\mcM$ and a right module $\mcN$, there's a bimodule $\mcM\otimes \mcN$ defined as $(\mcM\otimes \mcN)(A,B):=\mcM(A)\otimes \mcN({B})$ with structure maps:
    \begin{align}
        \mu_{\mcM\times\mcN}^{r|1|s}=\begin{cases}
            \mu_\mcM^{r|1}\otimes id & s=0, r>0\\
            id\otimes\mu_\mcN^{1|s} & r=0, s>0\\
            \mu_\mcM^{0|1}\otimes id+ id\otimes\mu_\mcN^{1|0} & r=0, s=0\\
            0 & r>0,s>0
        \end{cases}
    \end{align}
\end{defn}

\begin{defn}\label{abstract twist}
    Given a left module $\mcM$ and an object $X$, the \textbf{abstract twist} of $\mcM$ along $X$ is defined as $(T_X\mcM)(A):=\mcM(A)\oplus (A,X)\otimes\mcM(X)[1]$ with structure maps:
    \begin{align}
    \mu_{T_X\mcM}^{r|1}(a,b\otimes c)=(\mu_\mcM^{r|1}(a),\mu_\mcM^{r+1|1}(b,c)+b\otimes\mu_\mcM^{r|1}(c))
    \end{align}
\end{defn}

\begin{defn}
    Given any covariant $A_\infty$ functor $\phi:\mcF\rightarrow \mcF$, the graph bimodule $\mathfrak{G}_\tau$ is defined as $\gr{\phi}(A,B):=(\phi(B),A)$. The structure maps
\begin{align}
        \miu{k}{s}^{\phi}: ( A_{k-1},A_{k})\otimes...\otimes(\phi(B_0),A_0)\otimes...\otimes(B_{s},B_{s-1})\rightarrow (\phi(B_s),A_k)
\end{align}
are given by the composition
\begin{align}
\miu{k}{s}^\phi=\sum_{1\leq j\leq s}\mu_{k+s-j+2}\circ(\sum_{i_1+\cdots+i_j=j} id^{\otimes k+1}\otimes \phi_{i_1}\otimes\cdots\otimes \phi_{i_j})
\end{align}
\end{defn}

We can see that for the composed Dehn twist $\tau$, $\mathfrak{G}_\tau$ has a complicated structure map to perform calculations, since its structure maps involve the functor $\tau$. Here we construct using Lagrangian correspondences a bimodule that is quasi-isomorphic to $\mathfrak{G}_\tau$ and easier to work with.

The relation between Lagrangian correspondences $L_{01}\subset M_0\times M_1$ and functors from $H^0(Fuk(M_0))$ to $H^0(Fuk(M_1))$ was first written by Wehrheim and Woodward in \cite{WW1}. Then together with Ma'u they were able to enhance it to produce chain-level functors from $Fuk(M_0)$ to $Fuk(M_1)$ \cite{MWW}. Recall $\mcF^2$ is the subcategory of $Fuk(M\times M^-)$ generated by product Lagrangians and graphs of compactly supported exact symplectomorphism. From \cite{Mak-Wu} (also see \cite{Sei01}, \cite{Fuk17}, and \cite{Lek10}) we have an exact triangle in $\mcF^2$:

\begin{align}
    L_n\times L_n\rightarrow \Delta \rightarrow Gr(\tau_{n}^{-1})
\end{align}
Let $\mcF\bimodc\mcF$ be the category of $\mcF-\mcF$-bimodules. There's an $A_\infty$ functor $\Phi:\mcF^2\rightarrow \mcF\bimodc\mcF$. This functor is similar to the M'au-Wehrheim-Woodward functor in \cite{MWW}, except that the target is the category of bimodules instead of functors. It is also similar to the one constructed by Ganatra in \cite{Ganatra}, except that we are in the compact setting instead of the wrapped setting. The functor $\Phi$ takes the exact triangle above to the well-established exact triangle in $\mcF$-bimodule due to Seidel \cite{Sei08}:

\begin{align} \label{surgeryn}
        \bimod{L_n}{L_n}\xrightarrow{ev} \mcF_\Delta\rightarrow \mathfrak{G}_{\tau_n}
\end{align}
Where $ev:\bimod{L_n}{L_n}\rightarrow \mcF_\Delta$ is the map $ev_{r|1|s}:=\mu_{r+s+2}$.

There is an autoequivalence $(\ta{i}\times id)_\#$ on the subcategory of $\mcF$-bimodules generated by $\Phi(\mcF^2)$ that are induced by the symplectomorphisms $\ta{i}\times id$.  We define 
$$ev'_i:=(\ta{1}\times id\circ\cdots\circ \ta{i}\times id)_\#(ev)=(\ta{1}\times id)_\#\circ\cdots\circ (\ta{i}\times id)_\#(ev)$$
Applying $(\ta{1}\times id)_{\#}\circ\cdots\circ (\ta{n-1}\times id)_\#$ to (\ref{surgeryn}) we get an exact triangle:

\begin{align}\label{n-1}
        \bimod{\ta{1}\cdots\ta{n-1}L_n}{L_n}\xrightarrow{ev'_{n-1}} \mathfrak{G}_{\tau_{1}\cdots\tau_{n-1}}\rightarrow \mathfrak{G}_{\tau_{1}\cdots\tau_{n}}
\end{align}
 We introduce the following notation:
\begin{align}
    &\mathfrak{G}_i:=\mathfrak{G}_{\tau_{1}\cdots\tau_{i}}\\
    &\mcL_i:=\bimod{\ta{1}\cdots\ta{i-1}L_i}{L_i}
\end{align}
The structure maps of $\mcT_n:=\mcY^{l}_{\ta{1}\cdots\ta{n-1}L_n}$ can be computed easily using formulas for abstract twists in Definition \ref{abstract twist} which come from \cite{Sei08} and the fact that abstract twists of Yoneda modules are quasi-isomorphic to Dehn twists.

 Therefore by Definition \ref{product bimodule} the structure maps of $\mcL_i=\mcT_i\otimes \rmod{L_i}$ is given by
 \begin{align}
    \mu_{\mcL_i}^{r|1|s}=\begin{cases}
        \mu_{\mcT^i}^{r|1}\otimes id & s=0, r>0\\
            id\otimes\mu_{L_i}^{1|s} & r=0, s>0\\
            \mu_{\mcT_i}^{0|1}\otimes id+ id\otimes\mu_{L_i}^{1|0} & r=0, s=0\\
            0 & r>0,s>0
    \end{cases}
 \end{align}
Where $\mu_{\mcT_i}^{r|1}$ involves all possible contractions from the left.


We define 
\begin{align}
    \mcG_i&:=Cone(\mcL_i\xrightarrow{ev'_{i-1}} Cone( \mcL_{i-1}\xrightarrow{ev'_{i-2}} Cone(\cdots Cone(\mcL_1 \xrightarrow{ev} \Delta))))\\
    &=Cone(\mcL_i\xrightarrow{ev'_{i-1}} \mcG_{i-1})
\end{align}
Then we have from \eqref{n-1}
\begin{equation}\label{e:Gi}
    \mcG_i \simeq \mathfrak{G}_i
\end{equation}

 $\mcG_i$ is a more explicit model for $\mathfrak{G}_i$ since we can now describe the underlying vector spaces clearly:
\begin{align}\label{Gprime}
        \mcG_n(A,B)&=\bigoplus_{0\leq i\leq n-1} \mcL_{i+1}(A,B)[1]\\
        &=\bigoplus_{0\leq k\leq n}\bigoplus_{1\leq i_1< \cdots< i_k\leq n}CF(A,L_{i_1})\otimes\cdots\otimes CF(L_{i_k},B)[k]
    \end{align}

However, the structure maps of $\mcG_n$ are still implicit, since they include the maps $ev'_i$, whose formula is not explicit from the definition. In the following section, we will introduce a new collection of explicit maps $ev_i$ that represent $ev'_i$ up to multiplication by a nonzero constant. So we can write down the structure maps for $\mcG_n$ explicitly using $ev_i$.
\begin{rmk}
    In principle, it is possible to compute the formula for $ev'_n$ using Wehrheim-Woordward's result \cite{WWfiber} on the fibered Dehn twist $\tau_i\times id$
\end{rmk}

To describe the domain and codomain of the maps accurately we introduce the following notations for any bimodule $\mcB$: 
    \begin{align}
    &\mcB(\bfA_k,\bfB_s):=CF(A_k,A_{k+1})\otimes\cdots\otimes \mcB(A_0,B_0)\otimes\cdots\otimes CF(B_{s-1},B_{s}) \\
    &\mcnN{\summand}(A,B):=CF(A,L_{i_1})\otimes CF(L_{i_2},L_{i_1})\otimes\cdots\otimes CF(L_{i_r},B)\\
    &\nbf{\summand}:=\mu_{r+k+s+1}:\mcnN{\summand}(\bfA_k,\bfB_s)\rightarrow CF(A_k,B_s)\\
    &\bfnnn{\summand}{\summand}:=\sum_i id^{\otimes i}\otimes \mu_1\otimes id^{\otimes k-i}:\mcnN{\summand}(A_0,B_0)\rightarrow \mcnN{\summand}(A_0,B_0)\\
    &\bfnnn{\summand}{\rsubsummand}:=id\otimes\mu_{l+s+1}:\mcnN{\summand}(A_0,\bfB_s)\rightarrow\mcnN{\rsubsummand}(A_0,B_s)\\
    &\bfnnn{\summand}{\lsubsummand}:=\mu_{l'+k+1}\otimes id:\mcnN{\summand}(\bfA_k,B_0)\rightarrow\mcnN{\lsubsummand}(A_k,B_0)\\
    &\bfnnn{\summand}{\msubsummand}\\
    &:=id\otimes \mu_{l'-l+1}\otimes id:\mcnN{\summand}(A_0,B_0)\rightarrow\mcnN{\msubsummand}(A_0,B_0)
    \end{align}
    
    When $k\neq 0$, $\bfnn{\summand}$, $ \bfnnn{\summand}{\rsubsummand}$ and $\bfnnn{\summand}{\msubsummand}$ are $0$.
    When $s\neq0 $,  $\bfnn{\summand}$, $ \bfnnn{\summand}{\lsubsummand}$ and $\bfnnn{\summand}{\msubsummand}$ are $0$
\begin{defn}
    Under the identification of graded vector spaces 
    \begin{align}
        &\mcL_{i+1}(A, B)=\\
        &\bigoplus_{1\leq k\leq i+1}\bigoplus_{1\leq i_1< \cdots< i_k=i+1}\mcnN{\summand}(A,B)[k-1]\label{degree}
    \end{align}\\
    We define $ ev_i\in \bihom{\mcL_{i+1}}{\mcG_{\tau_1\cdots\tau_{i}}}$ as follows:
    \begin{align}\label{evn}
        \sum_{1<i\leq n+1} \sum_{1\leq i_1\leq\cdots\leq i_k=n+1}( (\sum_{\rthirdsum} \bfnnn{\summand}{\rsubsummand})+\nbf{\summand})
    \end{align}
\end{defn}
The following is the key proposition of our argument, which we will prove using induction in the next section.
\begin{prop}\label{Keyprop}
$[ev_i]=c[ev'_i]$ as elements in $ \biHom{\mcL_{i+1}}{\mcG_{i}}$ for some nonzero constants $c$.
\end{prop}

\begin{cor}
 There is a quasi-isomorphism
\begin{align}\label{beforechange} 
      \mcG_n\simeq Cone(\mcL_i\xrightarrow{ev_{i-1}} Cone( \mcL_{i-1}\xrightarrow{ev_{i-2}} Cone(\cdots Cone(\mcL_1 \xrightarrow{ev} \Delta))))
\end{align}
 
\end{cor}

The last step is to use the next lemma:


\begin{lem}\label{changecone0}
    Let $\mcX$, $\mcY$ and $\mcZ$ be $A_\infty$-bimodules, $g:\mcY\rightarrow\mcZ$ and $f:\mcX \rightarrow Cone(\mcY\xrightarrow{g}\mcZ)$ be bimodule morphisms. Then there exist unique maps $\tilde{f}:\mcX\rightarrow\mcY$ and $\tilde{g}:Cone(\mcX\xrightarrow{\tilde{f}}\mcY)\rightarrow \mcZ)$  such that this is a quasi-isomorphism:
    \begin{align}    \label{algebraic lemma}
        id:Cone(\mcX \xrightarrow{f} Cone(\mcY\xrightarrow{g}\mcZ))
    \cong Cone(Cone(\mcX[-1]\xrightarrow{\tilde{f}}\mcY)\xrightarrow{\tilde{g}} \mcZ)
    \end{align}

\end{lem}

\begin{proof}
    Let $X,Y,Z$ denote the underlying graded vector spaces of $\mcX, \mcY, \mcZ$ respectively. Clearly the underlying graded vector spaces of both side of (\ref{algebraic lemma}) agree and is equal to
    $$A:=X[1]\oplus Y[1] \oplus Z$$
    Under this identification of vector spaces, we have: 
   $$f=(f_1,f_2): X[1]\rightarrow Y[1]\oplus Z$$
    and  $$g:Y[1]\rightarrow Z$$
    Now define $$\tilde{f}:=f_1$$ $$\tilde{g}:=f_2+g$$
    We see $$d_A(x[1],y[1],z)=(d_Xx[1], d_Yy[1]+f_1x[1],d_Zz+f_2x[1]+gy[1])$$
    $$d_B(x[1],y[1],z)=(d_Xx[1],d_Yy[1]+\tilde{f}x[1],d_Zz+\tilde{g}(x[1],y[1]))$$
    $$=(d_Xx[1],d_Yy[1]+f_1x[1],d_Zz+f_2x[1]+gy[1])$$

    Indeed they agree. The compatibility with higher operations can be verified likewise. 
\end{proof}

Applying lemma \ref{changecone0} repeatedly on \eqref{beforechange} we can rearrange the iterated cone:
\begin{equation}\label{afterchange}\mcG_n\cong Cone(\cdots Cone(\mcL_n[-n+1]\xrightarrow{\widetilde{ev_{n-1}}} \cdots) \xrightarrow{\widetilde{ev_1}}\mcL_1) \xrightarrow{\widetilde{ev}} \Delta)
\end{equation}
where all of $\widetilde{ev_i}$ combined is essentially the same as all of $ev_i$ combined. For example, the component of $ev_{i_k-1}$ in (\ref{beforechange}) that maps $$CF(A,L_{i_1})\otimes...\otimes CF(L_{i_k},B)[k]$$ to $$CF(A,B)$$ is the same as the map $\widetilde{ev}$ in (\ref{afterchange}) restricted to $CF(A,L_{i_1})\otimes...\otimes CF(L_{i_k},B)$: They are both just $\mu_{k+1}$.

Now define $$\mcE_n:=Cone(\cdots Cone(\mcL_n[-n+1]\xrightarrow{\widetilde{ev_{n-1}}} \mcL_{n-1}[-n+2])\xrightarrow{\widetilde{ev_{n-2}}} \cdots) \xrightarrow{\widetilde{ev_1}}\mcL_1).$$
Its structure map is explicit since ${ev_i}$ and hence $\widetilde{ev_i}$ are explicit.
By construction, we have that $\mcE_n$ sits in an exact triangle
\begin{align}\label{mainCone}
    \mcE_n\xrightarrow{\widetilde{ev}} \Delta \rightarrow \gr{\ta{1}\cdots\ta{n}}
\end{align}
What's left to prove is Proposition \ref{Keyprop}, which we will prove in the next section using induction and a dimension argument.


\subsection{Proof of Proposition \ref{Keyprop} by induction}\label{degree section}

We explicitly describe the domain and codomain of $ev'_i$ first. Recall by (\ref{Gprime}):
\begin{defn}
    Let $\mcG_n$ be the bimodule whose underlying vector space is
    \begin{align}
        \mcG_n(A,B)&=\bigoplus_{0\leq i\leq n-1} \mcL_{i+1}(A,B)[1]\\
        &=\bigoplus_{0\leq k\leq n}\bigoplus_{1\leq i_1< \cdots< i_k\leq n}\mcnN{\summand}(A,B)[k]
    \end{align}
    where it is understood that when $k=0$, $\mcN^\emptyset(A,B)=CF(A,B)$. 
    Its structure maps is given by $A_\infty$ multiplications and $ev'_i$:
    \begin{align}
        &\sum_{0\leq k\leq n} \sum_{\secondsum}( \bfnn{\summand}
    +\sum_{\lthirdsum} \bfnnn{\summand}{\lsubsummand}
    + \sum_{\mthirdsum} \bfnnn{\summand}{\msubsummand})\label{structure terms}\\
    &+\sum_{0\leq k\leq n-1} ev'_{k}\label{morphism terms}
    \end{align}
    Where it is understood that when $k=0$, $ev'_0=ev$ is the map $ev:\bimod{L_1}{L_1}\rightarrow \mcF_\Delta$ in (\ref{surgeryn})
\end{defn}


\begin{defn}
    Under the identification of graded vector spaces 
    \begin{align}
        &\mcL_{i+1}(A, B)=\\
        &\bigoplus_{1\leq k\leq i+1}\bigoplus_{1\leq i_1< \cdots< i_k=i+1}CF(A,L_{i_1})\otimes\cdots\otimes CF( L_{i_k},B)[k-1]\label{degree2}
    \end{align}\\
    We define $ ev_i\in \bihom{\mcL_{i+1}}{\mcG_{\tau_1\cdots\tau_{i}}}$ as follows:
    \begin{align}\label{evn2}
        \sum_{1<i\leq n+1} \sum_{1\leq i_1\leq\cdots\leq i_k=n+1}( (\sum_{\rthirdsum} \bfnnn{\summand}{\rsubsummand})+\nbf{\summand})
    \end{align}
\end{defn}

Now we proceed by induction to prove proposition \ref{Keyprop}. The case when $i=1$ is exactly (\ref{surgeryn}). Assuming for $i<n$, $ev_i$ satisfy the hypothesis of the following lemma, then we can use the following dimension argument to conclude $[ev_i]=c[ev'_i]$ for $i<n$ for some nonzero constant $c$:

           

\begin{lem}\label{morphism2}
     Let $\theta\in hom_{\mcF-\mcF}(\mcL_{i+1},\mcG_i)$ be any bimodule \textit{pre-morphism} satisfying  
     \begin{enumerate}
            \item   $\theta$  of degree zero.
            \item  $\theta$  is closed (i.e it is a bimodule \textit{morphism}).
            \item  $[\theta]\neq 0$.
        \end{enumerate} then its class $[\theta]$ is sent under the isomorphism 
        \begin{align}
              &H_*(L_{i+1})\cong HF(L_{i+1},L_{i+1}) \cong HF(L_{i+1}\times L_{i+1},\Delta)\\
              &\cong HF(\tau_{1}\cdots\tau_{i}(L_{i+1})\times L_{i+1},Gr(\ti{i}\cdots\ti{1}))\\
              &\cong Hom_{\mcF-\mcF}(\mcL_{i+1},\mcG_i)\cong  Hom_{\mcF-\mcF}(\mcL_{i+1},\mcG_i)
        \end{align}
   
     to some nonzero elements in $H_0(L_i)$ 
\end{lem}
\begin{rmk}
    Note that when $d$ is odd this argument works fine with just $\Z/2$ grading. See Remark \ref{bigcomment}
\end{rmk}
This is because $ev'_i$ also satisfies the hypothesis of Lemma \ref{morphism2} since it is the image of a degree preserving homologically nontrivial morphism under the fully faithful functor $(\ta{1}\times id\circ\cdots\circ \ta{i}\times id)_\#$.

Therefore, by the induction hypothesis, when we are proving that $ev_n$ satisfies the hypothesis of the above lemma, we can replace all the $ev'_i$ in the structure maps of $\mcG_n$ by $ev_i$. 
 
First of all, note that the degree of $ev_n$ is $0$ by construction. Indeed we have:
\begin{align}
    \prescript{}{i_1\cdots i_k}{\mathbf{n}}^{0|1|0}=\mu_{k+1}
\end{align}
which has degree $1-k$ and maps $$CF(A,L_{i_1})\otimes\cdots\otimes CF( L_{i_k},B)[k-1]$$ into $(A,B)$. Also
\begin{align}
    \prescript{}{i_1\cdots i_k}{\mathbf{n}}^{0|1|0}_{i_1\cdots i_{k-l}}=\mu_{l+1}
\end{align}
Which has degree $1-l$ and maps $$CF(A,L_{i_1})\otimes\cdots\otimes CF(L_{i_k},B)[k-1]$$
into
$$CF(A,L_{i_{k-l}})\otimes\cdots\otimes CF(L_{i_k},B)[k-l]$$

The proof that  $ev_n$ is closed is a long computation which we retain to the next section. 
Once we prove that $ev_i$ is closed. The fact that it represents a nontrivial class follows by using the following basic algebraic lemma on ${ev_i}$ restricting to the submodule $\bimod{L_{i+1}}{L_{i+1}}\subset \mcL_i$, we see the map after the restriction is $ev$, which we know induces nontrivial maps in homology from \cite{Mak-Wu}.
     
    \begin{lem}
        Let $f:\mcB\rightarrow \mcB'$ be a bimodule morphism. Suppose that there exists a submodule $\Bar{\mcB}$ of $\mcB$ such that $(f_{|{\Bar{\mcB}}})_*:H_*(\Bar{\mcB})\rightarrow H_*(\mcB')$ is nontrivial, then $f_*$ is nontrivial.
    \end{lem}

    \begin{prop}
    The maps $ev_i$ are homologically non-trivial.
    \end{prop}
    \begin{proof}
    We just need to prove $(ev_i)^{0|1|0}$ induces a nontrivial map on homology. Note that $\bimod{L_{i+1}}{L_{i+1}}\subset \mcL_i$ is a sub-module and the restriction of $ev_i$ is the map $ev$ in (\ref{surgeryn}), which was shown to induce nontrivial map on homology
    \end{proof}



\subsection{$ev_i$ is closed}\label{computation section}
Recall by the argument using induction hypothesis in the last section, we can replace the maps $ev'_i$ in $\mu_{\mcG_n}$ by $ev_i$.  This allows us to prove that $ev_n$ is closed:
 
    \begin{lem}\label{keylemma} $ev_n$ is closed, that is, the following equality holds
        \begin{align}
        &\sum\limits_{i,j}\mu^{r-i|1|s-j}_{\mcG_n}\circ ev^{i|1|j}_n\\
        &=\sum\limits_{i,j}ev_n^{r-i|1|s-j}\circ id^{\otimes^{r-i}}\otimes\mu_{\mcL_{n+1}}^{i|1|j}\otimes id^{\otimes^{s-j}}\\
        &+\sum\limits_{i,j}ev_n^{r-i+1|1|s}\circ id^{\otimes^{r-i-j}}\otimes\mu^i_\mcA\otimes id^{\otimes^{s+j+1}}\\
        &+\sum\limits_{i,j}ev_n^{r|1|s-j+1}\circ id^{\otimes^{r+i+1}}\otimes\mu^j_\mcA\otimes id^{\otimes^{s-i-j}}
        \end{align}
    \end{lem}

    
\begin{proof}

        We verify this for $r=s=0$ first, which is to prove the following equation: 
        \begin{align}\label{010 equation}
              \mu^{0|1|0}_{\mcG_n}\circ ev^{0|1|0}_n=ev_n^{0|1|0}\circ \mu^{0|1|0}_{\mcL_{n+1}}
        \end{align}
             
        We just need to prove this for any given pair of summands that is the domain and codomain. That is, given $1\leq \summandnew \leq n$ and a subset $\{\subsummandnew\}\subset\{\summandnew\}$, we just need to prove the LHS and RHS of (\ref{010 equation}) agree on each component: $$\prescript{}{\summandnew}(\mu^{0|1|0}_{\mcG_n}\circ ev^{0|1|0}_n)_{\subsummandnew}=\prescript{}{\summandnew}(ev_n^{0|1|0}\circ \mu^{0|1|0}_{\mcL_{n+1}})_{\subsummandnew}$$ that maps the component $\mcN^{\summandnew}$ to $\mcN^{\subsummandnew}\subset\mcG_n$.

        Fixing index set $\{\summandnew\}$, a subset $\{\subsummandnew\}$ and elements 
        $$(c_1,\cdots,c_{m+1})\in \mcN^{\summandnew}$$
        we compute 
        \begin{equation}\label{LHS}
            \prescript{}{\summandnew}(\mu^{0|1|0}_{\mcG_n}\circ ev^{0|1|0}_n)_{\subsummandnew}(c_1,\cdots,c_{m+1})
        \end{equation}
        and \begin{equation}\label{RHS}
            \prescript{}{\summandnew}(ev_n^{0|1|0}\circ \mu^{0|1|0}_{\mcL_{n+1}})_{\subsummandnew}(c_1,\cdots,c_{m+1})
        \end{equation}

        First of all, observe that both (\ref{RHS}) and (\ref{LHS}) vanish unless $\{\subsummandnew\}\subset\{\summandnew\}$ if of one of the following form:
        \begin{enumerate}
            \item $\summandnew =i_1\cdots i_p\subsummandnew i_q \cdots i_m$ for some $p\geq1$ and $q\leq m$.
            \item $\summandnew=j_1\cdots j_p i_r\cdots i_{r+l} j_{p+1}\cdots j_l i_{q}\cdots i_m$ for some $p\geq 1,r>1,l\geq 0,q\leq m$.
            \item $\summandnew=\subsummandnew i_q\cdots i_m$ for some $q$.
            
        \end{enumerate}
        since $ev_i$ only involves $A_\infty$ multiplication from the right and $\mu_{\mcL_i}$ only involves $A_\infty$ multiplication from the left. Therefore, for all other types of ordered subsets $\subsummandnew$, there are no terms in $\mu^{0|1|0}_{\mcG_n}\circ ev^{0|1|0}_n$ or $ev_n^{0|1|0}\circ \mu^{0|1|0}_{\mcL_{n+1}}$ that map $\mcN^{\summandnew}$ to $\mcN^{\subsummandnew}$.

        Case 0: The empty subset corresponds to the summand $\mcN^{\emptyset}=\Delta$. In this case $\mu^{0|1|0}_{\mcG_n}\circ ev^{0|1|0}_n=ev_n^{0|1|0}\circ \mu^{0|1|0}_{\mcL_{n+1}}$ is equivalent to the $A_\infty$ relation of order $m+2$. 

        Case 1: 
        \begin{align}
            &\mu^{0|1|0}_{\mcG_n}\circ ev^{0|1|0}_n(c_1,\cdots,c_{m+1})\\
            =&(\bfnnnz{i_1\cdots i_{q-1}}{i_{p+1}\cdots i_{q-1}}\circ \bfnnnz{i_1\cdots i_m}{i_1\cdots i_{q-1}})(c_1,\cdots,c_{m+1})\\
            =&(\mu_{p+1}(c_1,\cdots,c_{p+1}),\cdots,\mu_{m-q+2}(c_{q},\cdots,c_{m+1}))
        \end{align} 

        \begin{align}
             &ev_n^{0|1|0}\circ \mu^{0|1|0}_{\mcL_{n+1}}(\elements)\\
             =&(\bfnnnz{i_{p+1}\cdots i_{m}}{i_{p+1}\cdots i_{q-1}}\circ \bfnnnz{i_1\cdots i_m}{i_{p+1}\cdots i_{m}})(\elements)\\
             =&(\mu_{p+1}(c_1,\cdots,c_{p+1}),\cdots,\mu_{m-q+2}(c_{q},\cdots,c_{m+1}))
        \end{align}
        This finishes the verification for Case 1.

        Case 2: 
        \begin{align}
            &\mu^{0|1|0}_{\mcG_n}\circ ev^{0|1|0}_n(\elements)\\
            =&(\bfnnnz{i_1\cdots i_{q-1}}{i_1\cdots i_{r}i_{r+l}\cdots i_{q-1}}\circ \bfnnnz{i_1\cdots i_m}{i_1\cdots i_{q-1}})(\elements)\\
            =&(\cdots,\mu_{l+2}(c_{p+1},\cdots,c_{p+l+2}),\cdots,\mu_{m-q+2}(c_q,\cdots,c_{m+1}))
        \end{align}

         \begin{align}
           &ev_n^{0|1|0}\circ \mu^{0|1|0}_{\mcL_{n+1}}(\elements)\\
           =&(\bfnnn{i_1\cdots i_{r}i_{r+l}\cdots i_m}{i_1\cdots i_{r}i_{r+l}\cdots i_{q-1}}\circ \bfnnn{i_1\cdots i_m}{i_1\cdots i_{r}i_{r+l}\cdots i_{m}})(\elements)\\
           =&(\cdots,\mu_{l+2}(c_{p+1},\cdots,c_{p+l+2}),\cdots,\mu_{m-q+2}(c_q,\cdots,c_{m+1}))
        \end{align}
        This finishes verification for Case 2.
        

         Case 3: \begin{align}\label{case3 1}
             &\mu^{0|1|0}_{\mcG_n}\circ ev^{0|1|0}_n(\elements)\\
             =&\sum_{0<l<m-q+1}\bfnnnz{i_1\cdots i_{m-l}}{i_1\cdots i_{q-1}}\circ \bfnnnz{i_1\cdots i_m}{i_1\cdots i_{m-l}}(\elements)\\
             &+\bfnnnz{i_1\cdots i_{q-1}}{i_1\cdots i_{q-1}}\circ\bfnnnz{\summand}{i_1\cdots i_{q-1}}\\
             =&\sum_{i<q} (c_1,\cdots,\mu_1(c_i),\cdots,\mu_{m-q+2}(c_q,\cdots,c_{m+1})\\
             &+(c_1,\cdots,\mu_1(\mu_{m-q+2}(c_q,\cdots, c_i,\cdots,c_{m+1})))\\
             &+\sum_{0<l<m-q+1} (c_1,\cdots,\mu_{m-l-q+2}(c_q,\cdots,\mu_{l+1}(c_{m-l+1},\cdots,c_{m+1})))
         \end{align}
         
         \begin{align}
             &ev_n^{0|1|0}\circ \mu^{0|1|0}_{\mcL_{n+1}}(\elements)\\
             =&\bfnnnz{i_1\cdots i_{m}}{i_1\cdots i_{q-1}}\circ \bfnnnz{i_1\cdots i_m}{i_1\cdots  i_{m}}(\elements)\label{case3 2}\\
                &+\sum_l\bfnnnz{i_1\cdots i_{r}i_{r+l}\cdots i_{m}}{i_1\cdots i_{q-1}}\circ \bfnnnz{i_1\cdots i_m}{i_1\cdots i_{r}i_{r+l}\cdots i_{m}}(\elements) \label{case3 3}\\
            =&\sum_{i<q} (c_1,\cdots,\mu_1(c_i),\cdots,\mu_{m-q+2}(c_q,\cdots,c_{m+1})\\
            &+\sum_{i\geq q} (c_1,\cdots,\mu_{m-q+2}(c_q,\cdots, \mu_1(c_i),\cdots,c_{m+1})\\
            &+\sum_l (c_1,\cdots,\mu_{m-l-q+1}(c_q,\cdots,\mu_{l+2}(\cdots),\cdots,c_{m+1}))
         \end{align} 
         We see that both sides agree by the $A_\infty$ relations of order $m-q+3$.

         For the case when $r>0$ and $s>0$ we proceed similarly. First, we list all the cases for each of the three maps $\mu^{r|1|s}_{\mcG_n}, \mu^{r|1|s}_{\mcL_{n+1}}$ and $ev^{r|1|s}_n$:

         \begin{align}\label{ev r1s}
            ev^{r|1|s}_n=\begin{cases}
                \sum\limits_{0\leq k \leq n+1}\sum\limits_{1\leq i_1\cdots i_k=n+1} \nbf{\summand}, r\neq 0, s\neq 0\\
                 \sum\limits_{0\leq k \leq n+1}\sum\limits_{1\leq i_1\cdots i_k=n+1} \nbf{\summand}, r\neq 0, s= 0\\
                  \sum\limits_{0\leq k \leq n+1}\sum\limits_{1\leq i_1\cdots i_k=n+1} (\nbf{\summand}+ \sum\limits_{0<l<k}\bfnnn{\summand}{i_1\cdots i_{k-l}}), r= 0, s\neq 0\\         
            \end{cases}
         \end{align}
          \begin{align}\label{mcg r1s}
            \mu^{r|1|s}_{\mcL_{n+1}}=\begin{cases}
                0, r\neq 0, s\neq 0\\
                  \sum\limits_{0\leq k\leq n+1}\sum\limits_{1\leq i_1\cdots i_k\leq n+1}\sum\limits_{0\leq l<k} \bfnnn{\summand}{i_l\cdots i_{k}}, r\neq 0, s= 0\\
                 \sum\limits_{0\leq k\leq n+1}\sum\limits_{1\leq i_1\cdots i_k\leq n+1} \bfnnn{\summand}{\summand},  r= 0, s\neq 0\\
            \end{cases}
         \end{align}
          \begin{align}\label{mce r1s}
            \mu^{r|1|s}_{\mcE_{n}}=\begin{cases}
                 \sum\limits_{0\leq k\leq n}\sum\limits_{1\leq i_1\cdots i_k\leq n} \nbf{\summand}, r\neq 0, s\neq 0\\
                  \sum\limits_{0\leq k\leq n}\sum\limits_{1\leq i_1\cdots i_k\leq n}(\sum\limits_{0\leq l<k} \bfnnn{\summand}{i_l\cdots i_{k}} + \nbf{\summand}), r\neq 0, s= 0\\
                 \sum\limits_{0\leq k\leq n}\sum\limits_{1\leq i_1\cdots i_k\leq n} ( \sum\limits_{0\leq l<k}\bfnnn{\summand}{i_1\cdots i_{k-l}}+ \nbf{\summand}),  r= 0, s\neq 0\\
            \end{cases}
         \end{align}
        Now fix any subset $\{j_1\cdots j_l\}\subset \{i_1\cdots i_m\}$ just like in the case when $r=s=0$, the maps in Lemma $\ref{keylemma}$ vanishes unless $\{j_1\cdots j_l\}$ is one of the four cases we listed there. We verify the equality in $\ref{keylemma}$ for these 4 cases one by one.

        Case 0: $\emptyset\subset \{i_1\cdots i_m\}$. This is equivalent to $A_\infty$-relation of order $m+r+s+2$.

        Case 1: $\summandnew =i_1\cdots i_p\subsummandnew i_q \cdots i_m$. In this case, all maps in Lemma $\ref{keylemma}$ vanish unless $r=s=0$. This is because when $\{\subsummandnew\}$ is not empty, all maps with $r\neq 0$ and $s\neq 0$ vanish. But then neither (\ref{ev r1s}.2)$\circ$ (\ref{mcg r1s}.2) nor (\ref{mce r1s}.3)$\circ$ (\ref{ev r1s}.3) maps $\mcN^{\summandnew}$ to $\mcN^{\subsummandnew}$.

        Case 2: $\summandnew=j_1\cdots j_p i_r\cdots i_{r'} j_{p+1}\cdots j_{m'} i_{q}\cdots i_m$. In this case, all maps in Lemma $\ref{keylemma}$ vanish unless $r=s=0$. This is because $\bfnnn{\summand}{\msubsummand}$ are $0$ whenever $r>0$ or $s>0$

        Case 3: We have 
        \begin{flalign}
            &\sum_{j}\mu^{0|1|s-j}_{\mcG_n}\circ ev^{0|1|j}_n(b_s,\cdots ,b_{j+1})(b_j,\cdots,b_1)(x_m,\cdots,x_0)&&\\
            =&\sum_j\sum_{p+q=l}\bfnnn{i_1\cdots i_{m-l}}{i_1\cdots i_{m-p-q}}\circ \bfnnn{i_1\cdots i_m}{i_1\cdots i_{m-p}}(b_s,\cdots ,b_{j+1})(b_j,\cdots,b_1)(x_m,\cdots,x_0)&&\\
            =&\sum_j\sum_{p+q=l} (\mu_{s-j+q+1}(b_s,\cdots,b_{j+1},\mu_{j+p+1}(b_j,\cdots b_1,x_m,\cdots,x_{m-q})x_{m-p-1},\cdots , x_{m-p-q}),x_{m-p-q-1},\cdots, x_0)&&\label{case3r1s1}
        \end{flalign}
          Where the sum is over $p+q=l$ such that $i_{m-p-q}=j_{m'}$.
        \begin{align}
            &\sum_j ev_n^{0|1|s-j}\circ \mu_{\mcL_{n+1}}^{0|1|j}(b_s,\cdots ,b_{j+1})(b_j,\cdots,b_1)(x_m,\cdots,x_0)\\
            =&\sum_j \bfnnn{i_1\cdots i_m}{i_1\cdots i_{m-l}}\circ \bfnnn{i_1\cdots i_m}{i_1\cdots i_m}(b_s,\cdots ,b_{j+1})(b_j,\cdots,b_1)(x_m,\cdots,x_0)\\
            =&\sum_j (\mu_{s-j+l+1}(b_s,\cdots,b_{j+1},\mu_{j+1}(b_j,\cdots,b_1,x_m)),x_{m-1},\cdots,x_{m-l}),x_{m-l-1},\cdots,x_0)\label{case3r1s2}
        \end{align}

        \begin{align}
            &\sum_{i,j} ev_n^{0|1|s-j+1}(b_s,\cdots,\mu_j(b_{i+j},\cdots, b_{i+1}),b_i,\cdots,b_1)(x_m,\cdots,x_0)\\
            =&\sum_{i,j} (\mu_{s-j+l+2}(b_s,\cdots,\mu_j(b_{i+j},\cdots, b_{i+1}),b_i,\cdots,b_1,x_m,\cdots,x_{m-l}),x_{m-l-1},\cdots,x_0)\label{case3r1s3}
        \end{align}
        We see that (\ref{case3r1s1})+(\ref{case3r1s2})+(\ref{case3r1s3})=0 by applying the $A_\infty$ relations separately for each $l$. This finishes the proof for Lemma \ref{keylemma} and hence Theorem \ref{mainmain}.

\end{proof}

\begin{proof}[proof of theorem \ref{mainmain}]
    The exact triangle is obtained by using lemma \ref{changecone0} on the right-hand side of (\ref{beforechange}) as indicated. The statement about the differential follows from our formula for $ev_i$.
\end{proof}

\printbibliography

\end{document}